\documentclass[11pt]{article}     
\pdfoutput=1 
\usepackage{graphicx}
\usepackage{url}
\usepackage{color}

\usepackage[mathscr]{euscript}		

\usepackage{dsfont}

\usepackage{psfrag}			
\usepackage{hyperref}

\usepackage{amsmath}
\usepackage{amsthm}
\usepackage{mathtools}
\usepackage{amssymb}
\usepackage{bbm}
\usepackage{dsfont}
\usepackage{MnSymbol} 
\usepackage{caption} 

\usepackage{anysize}

\usepackage{enumerate}
\usepackage{enumitem}

\usepackage{ulem}	
\usepackage{sidecap}    

\usepackage[top=.95in, bottom = 1.02 in, left=1in, right = 1in]{geometry}

\newcommand{\floor}[1]{{\lfloor #1 \rfloor}}
\newcommand{\ceil}[1]{{\lceil #1 \rceil}}
\newcommand{\Z}{\ensuremath{\mathbb{Z}}}
\newcommand{\N}{\ensuremath{\mathbb{N}}}
\newcommand{\R}{\ensuremath{\mathbb{R}}}

\newcommand{\E}{\ensuremath{\mathbb{E}}}
\renewcommand{\P}{\ensuremath{\mathbb{P}}}

\newtheorem{thm}{Theorem}[section]
\newtheorem{lemma}[thm]{Lemma}
\newtheorem{cor}[thm]{Corollary}
\newtheorem{prop}[thm]{Proposition}

\theoremstyle{definition}
\newtheorem{defn}[thm]{Definition}
\newtheorem{rem}[thm]{Remark}

\DeclareMathOperator{\intt}{int}
\DeclareMathOperator{\diam}{diam}

\newcommand{\dd}{\text{d}}

\newcommand{\mm}{\lsem m \rsem}

\newcommand{\be}{\begin{equation}}
\newcommand{\ee}{\end{equation}}

\numberwithin{equation}{section}

\definecolor{Red}{rgb}{1,0,0}
\definecolor{Blue}{rgb}{0,0,1}
\definecolor{Olive}{rgb}{0.41,0.55,0.13}
\definecolor{Yarok}{rgb}{0,0.5,0}
\definecolor{Green}{rgb}{0,1,0}
\definecolor{MGreen}{rgb}{0,0.8,0}
\definecolor{DGreen}{rgb}{0,0.55,0}
\definecolor{Yellow}{rgb}{1,1,0}
\definecolor{Cyan}{rgb}{0,1,1}
\definecolor{Magenta}{rgb}{1,0,1}
\definecolor{Orange}{rgb}{1,.5,0}
\definecolor{Violet}{rgb}{.5,0,.5}
\definecolor{Purple}{rgb}{.75,0,.25}
\definecolor{Brown}{rgb}{.75,.5,.25}
\definecolor{Grey}{rgb}{.7,.7,.7}
\definecolor{Black}{rgb}{0,0,0}

\newcommand{\ignore}[1]{{}}

\renewcommand{\P}{\ensuremath{\mathbb{P}}}

\title{Random Delta-Hausdorff-attractors }

\author{
Michael Scheutzow $^{1}$
\and Maite Wilke-Berenguer $^{1}$
}

\date{\today}

\begin{document}
\normalem

\maketitle
{\footnotesize

\thanks{$^1$ Institut f\"ur Mathematik, Technische Universit\"at Berlin, MA 7-5, Stra\ss e des 17. Juni 136, 10623 Berlin, Germany;\\
{\it e-mail:} \href{mailto:ms@math.tu-berlin.de}{ms@math.tu-berlin.de}, \href{mailto:wilkeber@math.tu-berlin.de}{wilkeber@math.tu-berlin.de}
}}

\maketitle

\begin{abstract}
Global random attractors and random point attractors for random dynamical systems have been studied for several decades. Here we introduce two intermediate 
concepts: $\Delta$-Hausdorff-attractors  are characterized by attracting all deterministic compact sets of Hausdorff dimension at most $\Delta$, where $\Delta$ is 
a non-negative number, while cc-attractors attract all countable compact sets. We provide two examples showing that a given random dynamical system may have various different $\Delta$-Hausdorff-attractors for different values 
of $\Delta$. It seems that both concepts are new even in the context of deterministic dynamical systems.
\end{abstract}

\noindent \textbf{Mathematics Subject Classification (2010):} 37H99; 37H10, 37B25, 37C70. 

\noindent \textbf{Keywords:} Random attractor, Delta-attractor, pullback attractor, weak attractor, forward attractor, cc-attractor, 
random dynamical system, Hausdorff dimension.


\section{Introduction} \label{sec:intro}
\emph{Global} random attractors as well as random \emph{point} attractors for random dynamical systems are established notions having been studied for several decades. They are, in a sense, different ends to a spectrum and we here introduce an intermediary concept, namely random $\Delta$-Hausdorff-attractors. A random $\Delta$-Hausdorff-attractor for a random dynamical system (RDS) taking values in a metric space $(E,\dd)$ is characterized by attracting all deterministic compact sets of Hausdorff dimension at most $\Delta$, where $\Delta$ is a non-negative number. As in the case of global or point attractors the attraction can be 
in the pullback, forward or weak sense (precise definitions are provided in the next section). We also introduce random cc-attractors which are required to attract 
all deterministic countable compact sets and as such are located `in between' the concepts of point and $\Delta$-Hausdorff-attractors. While global random attractors were introduced in \cite{CF94} and point attractors first appear in \cite{Crauel_PointVsSet} and both have been thoroughly discussed in many papers thereafter, the concepts of $\Delta$-Hausdorff-attractor and cc-attractor to the best of our knowledge seem to be new even in the deterministic case. 

In order to illustrate the usefulness and naturalness of the concepts we provide two examples. The first is a product of one-dimensional martingale diffusions on $[0,1]^d$. For it, we 
identify different kinds of  attractors of the associated RDS as well as of its time reversal. The second example is more challenging: 
its state space is the simplex $S^{m-1}$ and its dynamics are described by iterated independent Volterra polynomial stochastic operators - a concept originating in the mathematical modeling of population genetics. We are not able to identify all $\Delta$-Hausdorff-attractors for each value of $\Delta$ in this case, but we provide an example of a non-trivial $\Delta$-Hausdorff-attractor for a $\Delta>0$ that coincides with the point, but differs from the global attractor. It should be noted that the second example can be found in the dissertation of one of the authors, cf. \cite{MyDiss}.

As both the cc- and the $\Delta$-Hausdorff-attractors seem to be new concepts a myriad of questions arise. One could, for example, ask for criteria to determine the existence of cc- or $\Delta$-Hausdorff-attractors, that differ from the point or the global attractor (or from each other for different values of $\Delta$). Additional examples including phenomena like $\Delta$-Hausdorff-attractors that are themselves of (a non-integer) Hausdorff dimension $\Delta$ or comparative studies in the deterministic set-up would help deepen the understanding. 


\section{Notation and Definitions} \label{sec:prelim}
We begin with a brief introduction to the basic notions of random dynamical systems (RDS) and the notions of attractors. 
\begin{defn}
Let $(\Omega, \mathcal F, \P)$ be a probability space, $\mathbb T_1 \in \{\R, \Z\}$ endowed with its Borel-$\sigma$-algebra $\mathcal B(\mathbb T_1)$. Denote by $(\vartheta_t)_{t \in \mathbb T_1}$ a family of jointly measurable maps such that
\begin{enumerate}[label=\roman*]
 \item $\forall\, t \in \mathbb T_1$: $\P\circ\vartheta_t^{-1}=\P$ \label{enum:thetainv}
 \item $\vartheta_0 = \text{Id}_{\Omega}$ and $\forall\, t,s \in \mathbb T_1:$ $\vartheta_{t+s} = \vartheta_t\circ \vartheta_s$.\label{enum:thetagroup}
\end{enumerate}
Then $(\Omega, \mathcal F, \P, (\vartheta_t)_{t \in \mathbb T_1})$ is called a \emph{metric dynamical system (MDS)}. 
Now let $(E,\dd)$ be a metric space and $\mathbb T_2 \in \{\R^+_0, \R, \N_0, \Z\}$ such that $\mathbb T_2 \subseteq \mathbb T_1$, endowed with their Borel-$\sigma$-algebras $\mathcal B(E)$ and $\mathcal B(\mathbb T_2)$ respectively. A measurable map $ \varphi: \mathbb T_2\times\Omega\times E \rightarrow E,\; (t,\omega,x)  \mapsto \varphi(t,\omega)x$  with the properties
\begin{enumerate}[label=\roman*,resume]
 \item $\forall\, \omega \in \Omega:$ $\varphi(0,\omega) = \text{Id}_{E}$ and 
 \item $\forall\, t,s \in \mathbb T_2:$ $\varphi(t+s,\omega) = \varphi(t, \vartheta_s\omega) \circ \varphi(s, \omega)$, \label{enum:phicocycle}
\end{enumerate}
is called a \emph{cocycle} and the ensemble $(\Omega, \mathcal F, \P, (\vartheta_t)_{t \in \mathbb T_1}, \varphi)$ is a \emph{random dynamical system (RDS)}.
\end{defn}

 If $\mathbb T_2 \in \{\R, \Z\}$, for all $t \in \mathbb T_2$ and $\omega \in \Omega$ the map $\varphi(t,\omega)$ is a bimeasurable bijection of $E$ and $  \varphi(t,\omega)^{-1}=\varphi(-t,\vartheta_t\omega)$.
 Moreover, the mapping $(t,\omega,x) \mapsto \varphi(t,\omega)^{-1}x$ is $\mathcal B(\mathbb T_2) \otimes F \otimes \mathcal B(E)$-$\mathcal B(E)$-measurable. (See Theorem 1.1.6 in \cite{Arnold_book}.) 
 
 For any $x \in E$ and any non-empty $B \subseteq E$ set $  \dd(x,B):= \inf_{y \in B} \dd(x,y)$ and for any other non-empty $A \subseteq E$ and bounded $B$ define $\dd(A,B):= \sup_{x \in A}\dd(x,B) = \sup_{x \in A}\inf_{y \in B} \dd(x,y)$ the Hausdorff semi-distance induced by $\dd$. Using the same letter for the metric and for the Hausdorff semi-distance induced by it should not cause confusion and will be done so consistently for any arising metric. For a set $A \subseteq E$ and a $\delta > 0$, set $A^{\delta}:=\{x \in E \mid \dd(x,A)<\delta\}$. 
 
Let $\mathcal P(E)$ denote the power-set of $E$. A set-valued map $A:\Omega\rightarrow \mathcal P(E)$ is called a  \emph{compact random set}, if for each $\omega \in \Omega$ $A(\omega)$ is compact and for each $x \in E$, the map $ \omega \mapsto d(x,A(\omega))$ is measurable. Such a set is said to be \emph{(strictly) $\varphi$-invariant} for the RDS $\varphi$, if for every $t \in \mathbb T_2$
 \begin{align*}
  \varphi(t,\omega)A(\omega) = A(\vartheta_t\omega)\qquad \text{ for } \P\text{-almost all } \omega \in \Omega.
 \end{align*}

 We are now ready to define the notion of an attractor of an RDS.
 
 \begin{defn}\label{def:att}
  Let $(\Omega, \mathcal F, \P,(\vartheta_t)_{t \in \mathbb T}, \varphi)$ be a random dynamical system over the metric space $(E,\dd)$. Let $\mathcal C\subseteq \mathcal P(E)$ be an arbitrary non-empty subset of the power-set of $E$ and $A$ a compact random set that is strictly $\varphi$-invariant.
  \begin{enumerate}
   \item $A$ is called a \emph{forward $\mathcal C$-attractor}, if for every $C \in \mathcal C$
   \begin{align*}
    \lim_{t \rightarrow \infty} \dd(\varphi(t,\omega)C,A(\vartheta_t\omega)) = 0 \qquad \P\text{-a.s.}
   \end{align*}
    \item $A$ is called a \emph{pullback $\mathcal C$-attractor}, if for every $C \in \mathcal C$
   \begin{align*}
    \lim_{t \rightarrow \infty} \dd(\varphi(t,\vartheta_{-t}\omega)C,A(\omega)) = 0 \qquad \P\text{-a.s.}
   \end{align*}
   \item $A$ is called a \emph{weak $\mathcal C$-attractor}, if for every $C \in \mathcal C$
   \begin{align*}
    \lim_{t \rightarrow \infty} \dd(\varphi(t,\omega)C,A(\vartheta_t\omega)) = 0 \qquad \text{in probability.}
   \end{align*}
  \end{enumerate}
 \end{defn}
 
 Since in general the subsets of $\Omega$ under consideration above need not be measurable, the statements are to be understood in the sense that there exist suitable measurable sets containing, respectively being contained in the sets above. The notions of \emph{forward} and \emph{pullback} attractor coincide when the convergence is considered to be \emph{in probability} only, hence the absence of such a denomination in the definition of a weak attractor. Certain families of sets $\mathcal C$ are of particular interest and their attractors are thus referred to by specific terms:  If $\mathcal C = \mathcal K:=\{K \subset E\mid K\neq \emptyset\text{ is compact }\}$, the set of \emph{all} compact subsets of $E$, then $A$ is called  \emph{global} attractor. On the other hand, if  $\mathcal C = \{\{x\}\mid x \in E\}$, then $A$ is called \emph{point} attractor. Several of these notions were introduced in \cite{Crauel_PointVsSet, CF94} and \cite{Ochs}. For a comparison of different concepts of attractors we recommend \cite{ScheutzowComparison} as well as \cite{Crauel_PointVsSet}.
 
 However, as illustrated by the examples in Sections \ref{sec:Ex1} and \ref{sec:Ex2}, in some cases a notion `in between' these two concepts arises naturally. To define this, we recall the definition of the \emph{Hausdorff dimension} $\dim_H$ on a metric space $(E,\dd)$. For $H \subset E$ a sequence $E_1, E_2, \ldots$ of subsets of $E$ is called a \emph{cover of $H$}, if $ H \;\subseteq \bigcup_{i \in \N}E_i.$ 
For every $\delta \geq0$ and $\varepsilon >0$ define the \emph{$\delta$-Hausdorff-measure} of $H$ as 
 \begin{align*}
   \mathcal H^{\delta}(H) :=\sup_{\varepsilon>0}\inf\left\{\sum_{i \in \N}\diam(E_i)^{\delta} \,\big\mid\,E_1, E_2, \ldots \text{cover of } H,\forall\, i \in \N:\; \diam(E_i)<\varepsilon\right\}.
 \end{align*}
The \emph{Hausdorff dimension} of $H$ is then given by $\dim_H(H):=\inf\{\delta\mid\mathcal H^{\delta}(H)=0\}$. Without loss of generality we may assume the sets in the covers to be open (see for example  \cite{HausdorffDim}).

 \begin{defn} In the set-up of Definition \ref{def:att} if $\mathcal C = \{ C \in \mathcal K\mid C \text{ is countable}\}$, then $A$ is called a \emph{cc-attractor}. And for any $\Delta\geq 0$, if $\mathcal C = \mathcal C_{\Delta}:=\{K \in \mathcal K\mid \dim_H(K)\leq \Delta\}$, then $A$ is called a \emph{$\Delta$-Hausdorff-attractor}.
 \end{defn}

 Given the multitude of terms in the precise description of an attractor (`forward', `pullback', `weak', `global', `point'...) we will omit the addendum of `Hausdorff' and simply speak of a \emph{$\Delta$-attractor}, in particular when $\Delta$ is fixed. 
 
 Obviously, a global pullback attractor is a pullback $\Delta$-attractor for each $\Delta \ge 0$, each pullback $\Delta$-attractor is a pullback $\Delta'$-attractor 
if $\Delta' \le \Delta$, each pullback $\Delta$-attractor is a pullback cc-attractor and a pullback cc-attractor is a pullback point attractor. The same 
holds true with `pullback' replaced by either `forward' or `weak'. Note that the concepts of global, cc- and point attractors only depend on the topology 
of $(E,\dd)$ and not on the chosen metric $\dd$. This is generally not true for $\Delta$-Hausdorff-attractors. 
 
 It was proven in \cite{SbN}, Lemma 1.3, that the global weak attractor of an RDS that is continuous in the space-variable is $\P$-almost surely unique. Hence, if they exist, the  global forward attractor and the  global pullback attractor must coincide and they are unique $\P$-almost surely. (However, they need not exist.)
 Since uniquenes does not necessarily hold for other than global  attractors, we introduce the notion of a \emph{minimal} attractor.
 \begin{defn}
   A  $\mathcal C$-attractor for $\varphi$ is called \emph{minimal}, if it is contained in any other $\mathcal C$-attractor of the same type. 
 \end{defn}
 It was recently proven in \cite{ScheutzowCrauel}, that such a minimal attractor exists in the  pullback or weak sense, given \emph{any} such attractor exists. However, a minimal attractor need not exist in the forward sense.


\section{The first example} \label{sec:Ex1}
Our first example to illustrate the notion of $\Delta$-attractors is based on the SDE 
\begin{align}\label{eq:SDEEx1}
 \dd X^i(t) = X^i(t)(1-X^i(t))\dd W^i(t), \qquad i = 1, \ldots, d
\end{align}
on the space $E=[0,1]^d$ (equipped with the metric $\dd$ induced by the uniform norm), for some $d \in \N$. Here $W=(W^1, \ldots, W^d)$ is a two-sided $d$-dimensional Wiener process. 

Let the probability space $(\Omega,\mathcal F, \P)$ be the $d$-dimensional two-sided Wiener space and complete it to an MDS with the Wiener shift $(\vartheta_t)_{t \in \R}$ (i.e. $(\vartheta_t \omega)_s = \omega_{s+t}-\omega_t$). Theorem 2.3.39 in \cite{Arnold_book} guarantees the existence of a (two-sided) RDS $\varphi$ on $[0,1]^d$ over $(\Omega,\mathcal F, \P, (\vartheta_t)_{t \in \R})$  solving \eqref{eq:SDEEx1} for $t \geq 0$ and that is jointly continuous in time and space. Furthermore it implies that its inverse flow $\bar \varphi(t,\omega) := (\varphi(t,\vartheta_{-t}\omega))^{-1}$ on $[0,1]^d$ over $(\Omega,\mathcal F,\P, (\bar \vartheta_t)_{t \in \R})$ with $\bar \vartheta_t = \vartheta_{-t}$ exists, is jointly continuous in time and space and solves 
\begin{align}\label{eq:SDEEx1_inv}
 \dd Y^i(t) = Y^i(t)(1-Y^i(t))(1-2Y^i(t))\dd t - Y^i(t)(1-Y^i(t))\dd W^i(t), \qquad i = 1, \ldots, d,
\end{align}
where \eqref{eq:SDEEx1_inv} is also to be understood as a forward It\=o SDE.

Let us begin with some rather straightforward observations regarding attractors of the forward RDS $\varphi$. It is easily verified that $[0,1]^d$ is the \emph{global forward} and \emph{global pullback} attractor for $\varphi$ on $[0,1]^d$. In addition, note that for each $x \in E$ each coordinate of $t \mapsto \varphi(t,\cdot)x$ is a martingale and converges $\P$-a.s.\ to $0$ or $1$ independently of the other coordinates, whence we can conclude that $\{0,1\}^d$ is a forward point-attractor of $\varphi$. Since all elements of $\{0,1\}^d$ are fixed points for $\varphi$, this is even the \emph{minimal forward point-attractor}.

 This example, however, also allows for a precise description of its minimal (forward) $\Delta$-attractors. For any $\alpha \in \{-1,0,1\}^d$, define $\Gamma_{\alpha}:= \{x \in [0,1]^d \mid \forall \, i =1, \ldots, d:\, \alpha_i \neq 0 \Rightarrow x^i=(1+\alpha^i)/2\}$. We will call this a \emph{face} of $[0,1]^d$.  For any $m \leq d$ then set $H^m := \bigcup_{\Vert \alpha \Vert = d-m} \Gamma_{\alpha}$ which is the union of all $m$-dimensional faces. 

\begin{thm}\label{thm:delta_RDS1}
 For any $m = 0, \ldots, d-1$ the set $H^m$ is the \emph{minimal forward $\Delta$-attractor} and the \emph{minimal weak $\Delta$-attractor} of $\varphi$ on $[0,1]^d$ 
for every $\Delta \in [m, m+1[$.
\end{thm}

\begin{rem}
The identification of the minimal \emph{pullback} $\Delta$-attractors seems to be tricky. Let us just point out that Theorem \ref{thm:delta_RDS1} will definitely not hold with `forward' 
replaced by `pullback'. Consider the case $d=1$. Theorem \ref{thm:delta_RDS1} tells us that the minimal forward $\Delta$-attractor of $\varphi$ is $\{0,1\}$ if $\Delta \in [0,1[$. The pullback point attractor of $\varphi$ is however the whole space $E=[0,1]$: loosely speaking, for each $x \in ]0,1[$, $\varphi(t,\theta_{-t}\omega)x$ will sometimes be very close to 1 and sometimes very close to 0 when $t$ becomes large taking all intermediate values in between by continuity. 
\end{rem}

For the inverse flow $\bar \varphi$ on the other hand, we are able to obtain a complete characterization of all kinds of attractors introduced above.

As in the case of $\varphi$, the \emph{global forward attractor} and the \emph{global pullback attractor} of the inverse flow $\bar\varphi$ on $[0,1]^d$ are given by the complete space itself. The following two theorems give a detailed description of the remaining types of attractors, but require the introduction of some notation first. For $i = 1, \ldots, d$, $j \in \{ 0,1\}$ 
and $\omega \in \Omega$ define $M^i_j(\omega):=\{x \in [0,1]^d\mid \lim_{t \rightarrow \infty}\varphi^i(t, \omega)x^i = j\}$ and set
\begin{align}\label{eq:defb}
b^i(\omega):= \sup\{x^i\mid x \in M^i_0\},\qquad \text{and} \qquad b:=(b^1, \ldots, b^d). 
\end{align}
 By monotonicity and the componentwise convergence to $\{0,1\}$ observed above we also have $b^i=\inf\{x^i\mid x \in M^i_1\}$ $\P$-a.s.\ for any $i=1.\ldots,d$. Denote by $P_{\alpha}$ the orthogonal projection on the face $\Gamma_{\alpha}$.

\begin{thm}\label{thm:point_RDS1_inv}
 The \emph{minimal forward point attractor}, the \emph{minimal pullback point attractor} and the \emph{minimal weak point attractor} of $\bar \varphi$ on $[0,1]^d$ are all given by $A:=\{P_{\alpha}(b)\mid \alpha\in\{-1,0,1\}^d\}$, for $b$ defined in \eqref{eq:defb}. 
 Furthermore, $b$ is uniformly distributed on $]0,1[^d$.
\end{thm}
Note that $A$ consists of  exactly one point per face of $[0,1]^d$ and each such point $P_{\alpha}(b)$ is uniformly distributed on its corresponding face $\Gamma_{\alpha}$. Any `larger' attractor of $\bar\varphi$ on $[0,1]^d$, however, is given by the complete space itself, as we see in the following theorem.

\begin{thm}\label{thm:att_RDS1_inv}
 The \emph{minimal weak cc-attractor} of $\bar\varphi$ on $[0,1]^d$ is given by $[0,1]^d$ itself. 
\end{thm}
This allows us to conclude that $[0,1]^d$ is thus also the \emph{(minimal) $\Delta$-attractor} for any $\Delta \in [0,d]$ (in the pullback, forward and weak sense). 

The following observations will be helpful in the proofs of the theorems above.

\begin{lemma}\label{lem:b}
 The (random) point $b$ given in \eqref{eq:defb} repells sets in $[0,1]^d$ in the sense that
 \begin{align*}
  \P\left(\forall\, \delta>0:\; \lim_{t\rightarrow\infty}d\left(\varphi(t, \cdot)B_{\delta}(b)^c,H^{d-1}\right) = 0 \right) = 1.
 \end{align*}
\end{lemma}

\begin{proof}
It suffices to make the following simple observation:
\begin{align*}
 d\left(\varphi(t, \cdot)B_{\delta}(b)^c,H^{d-1}\right) 	& \leq \max_{i = 1, \ldots, d}\left\{\varphi^i(t, \cdot)(b^i(\cdot)-\delta), 1-\varphi^i(t, \cdot)(b^i(\cdot)+\delta)\right\}\\
								& \longrightarrow 0, \qquad \text{for all $\delta>0\;\, \P$-a.s. by construction of $b$.}
\end{align*}
\end{proof}

Theorem \ref{thm:point_RDS1_inv} is a direct consequence of the following partially stronger observation, that considers attractors of $\bar \varphi$ as an RDS \emph{on $]0,1[^d$}.

\begin{lemma}\label{lem:synch}
 The (random) point $b$ defined in equation \eqref{eq:defb} is the \emph{global pullback attractor} of $\bar \varphi$ as an RDS on $]0,1[^d$. Furthermore $b$ is uniformly distributed on $]0,1[^d$.
\end{lemma}
The phenomenon of a global attractor consisting of a single point is often referred to as \emph{synchronization}, see for example  \cite{SbN}. 
\begin{proof} The proof consists of two parts. First we prove the existence of a global weak attractor consisting of a single point that is uniformly distributed on $]0,1[^d$. Then, using Lemma \ref{lem:b}, we prove that $b$ is indeed the  global pullback attractor and must therefore coincide with the global weak attractor.

 The solution to \eqref{eq:SDEEx1_inv} clearly consists of $d$ independent copies of a solution to the one-dimensional SDE
  \begin{align*}
   \dd Y^1(t) = Y^1(t)(1-Y^1(t))(1-2Y^1(t))\dd t - Y^1(t)(1-Y^1(t))\dd W^1(t),
  \end{align*}
  whence we analyse the properties of this diffusion. Since its scale function $p$ and normalized speed measure $m$ are given by
  \begin{align*}
   p(y) = \frac{1}{16}\left(2\ln\left(\frac{y}{1-y}\right)-\frac{1-2y}{y(1-y)}\right) \quad \text{ and } \quad \dd m(y)=1\dd y
  \end{align*}
 respectively, we see that according to Feller's boundary classification (cf. \cite{EK}, Section 8.1) the boundaries 0 and 1 are both natural and thus inaccessible. In particular, when started in $]0,1[$, the diffusion cannot reach $\{0,1\}$ in finite time ($\P$-a.s.) justifying the restriction of $\bar\varphi$ to $]0,1[^d$. Theorem 7 in \cite{Mandl_book} implies the weak convergence of the transition probabilities of $Y^1$ to $m$, the uniform distribution on $]0,1[$. Since all its coordinates are independent, we immediately obtain the weak convergence of the transition probabilities of $Y=(Y^1, \ldots, Y^d)$ to $m^{\otimes d}$, i.e.\ the uniform distribution on $]0,1[^d$.
 Furthermore, as $\bar\varphi$ is order-preserving  Theorem 1 in \cite{ChSch04} allows us to conclude the existence of a random point $a \in ]0,1[^d$ uniformly distributed on $]0,1[^d$ such that $\{a\}$ is the weak global attractor of $\bar\varphi$ on $]0,1[^d$. 
 
 Now let us turn our attention to $b$. Recall that $\bar\varphi(t, \omega) = (\varphi(t,\vartheta_{-t}\omega))^{-1}$ for any  $\omega \in \Omega$ and $t \geq 0$. It is easy to check, that $\{b\}$ is invariant under $\bar \varphi$. Let $B \subset ]0,1[^d$ be compact. Therefore there exists an $\varepsilon>0$ such that $B\subseteq [\varepsilon,1-\varepsilon]^d$. Lemma \ref{lem:b} then implies that $\P$-a.s.
 \begin{align*}
  \forall\, \delta>0\; \exists\, T\;\forall\, t\geq T:\; B\subseteq [\varepsilon,1-\varepsilon]^d\subseteq \varphi(t,\cdot)B_{\delta}(b(\cdot)).
 \end{align*}
 Since for any $\omega \in \Omega$
 \begin{align*}
  &\forall\, \delta>0\; \exists\, T\;\forall\, t\geq T:\; B\subseteq \varphi(t,\omega)B_{\delta}(b(\omega)) \\
  \Leftrightarrow &\forall\, \delta>0\; \exists\, T\;\forall\, t\geq T:\; \bar\varphi(t,\bar\vartheta_{-t}\omega)B\subseteq B_{\delta}(b(\omega))\\
  \Leftrightarrow &\lim_{t\rightarrow \infty} \dd(\bar\varphi(t,\bar\vartheta_{-t}\omega)B,\{b(\omega)\}) = 0
 \end{align*}
we have proven that $\{b\}$ is the  global pullback attractor. As such, it must coincide with the weak global attractor $\{a\}$ and is thus uniformly distributed on $]0,1[^d$.

\end{proof}
It should be remarked that the definition of a weak (global) attractor in \cite{ChSch04} differs from the one in the present work in that it requires the attraction of all \emph{bounded} sets in probability. However, the cited Theorem 1 does yield a sufficient result. 

We now turn to the proof of Theorem \ref{thm:att_RDS1_inv}.
\begin{proof}[Proof of Theorem \ref{thm:att_RDS1_inv}]
 We first prove the claim in the case of $d=1$ and for simplicity write $Y:=Y^1$. We will consider a transformation of $Y$ to obtain a process taking values on all of $\R$. Define $f: [0,1] \rightarrow \R, y \mapsto -\ln\left(\frac{y}{1-y}\right)$ and observe that
 \begin{align*}
  f^{-1}(z) = \frac{e^{-z}}{1+e^{-z}},\;\;z \in \R;\quad f'(y)=-\frac{1}{y(1-y)},\quad f''(y) = \frac{1-2y}{y^2(1-y)^2},\;\;y \in ]0,1[.
 \end{align*}
For $Z:=f\circ Y$ \eqref{eq:SDEEx1_inv} and  It\=o's Lemma yields
\begin{align}\label{eq:RDS1_inv_transform}
 \dd Z(t) = \frac{1}{2}\frac{1-e^{-Z(t)}}{1+e^{-Z(t)}} \dd t + \dd W(t).
\end{align}
We denote the corresponding RDS on $\R$ by $\psi$. Consider a sequence $(z_i)_{i\in \Z}$ of initial values such that $z_i < z_{i+1},\,i \in \Z$, 
$\lim_{\vert i\vert\rightarrow \infty} (z_{i+1}-z_{i})=0$ and $\lim_{\vert i\vert \rightarrow \infty} \vert z_i\vert = \infty$. Observe that for the drift $\tilde b(z):= (1-e^{-z})/(2+2e^{-z})$ we have $\tilde b'(z) < 0$ for all $z \in \R$, hence $t \mapsto \psi(t,\cdot)z_{i+1} - \psi(t,\cdot)z_i$ is decreasing for every $i \in \Z$, $\P$-almost surely. We may estimate
\begin{align*}
  \sup_{i \in \Z} \vert \varphi(t,\cdot)f^{-1}(z_{i+1}) -& \varphi(t,\cdot)f^{-1}(z_{i})	\vert  \leq \frac{1}{4} \sup_{i \in \Z} \vert \psi(t,\cdot)z_{i+1} - \psi(t,\cdot)z_{i}\vert  \\
												& \leq \frac{1}{4} \underbrace{ \sup_{\stackrel{i \in \Z}{ \vert i \vert\geq J}}\vert \psi(t,\cdot)z_{i+1} - \psi(t,\cdot)z_{i}\vert}_{=:A_J(t,\cdot)}  + \frac{1}{4} \underbrace{\sup_{\stackrel{i \in \Z}{ \vert i\vert<J}} \vert\psi(t,\cdot)z_{i+1} - \psi(t,\cdot)z_{i}\vert}_{=:B_J(t,\cdot)}
\end{align*}
for any $J \in \N$. By construction we have $\lim_{J\rightarrow \infty}\sup_{t \geq 0}B_J(t,\cdot)=0$. In addition, translating Lemma \ref{lem:synch} to $\psi$ and taking into consideration the monotonicity of $\psi$ observed before yields $A_J(t,\cdot) \rightarrow 0$, $t\rightarrow \infty$, $\P$-a.s.\ for every $J \in \N$. 

Together we therefore see
\begin{align*}
 \sup_{t\geq T}\sup_{i \in \Z} \vert \varphi(t,\cdot)f^{-1}(z_{i+1}) - \varphi(t,\cdot)f^{-1}(z_{i})  \vert \rightarrow 0, \quad t\rightarrow \infty \qquad \P\text{-a.s.}
\end{align*}
We are now ready to conclude the following: $B:=\{0,1\}\cup\{f^{-1}(z_i), i \in \Z\}\subseteq [0,1]$ is a cc-set such that for any choice of $0 \leq x <y\leq 1$ and $\P$-almost all $\omega \in \Omega$ there exists a $t_0=t_0(x,y,\omega)<\infty$ such that for all $t \geq t_0$ we have $\varphi(t,\omega)B  \cap [x,y] \neq \emptyset$. Hence the minimal weak cc-attractor must be equal to $[0,1]$.

For general $d\geq 1$ we now simply consider $B_d:=B^d$. Since the observations above still hold componentwise, we know that for any $0\leq x_i<y_i\leq1$, $i =1, \ldots, d$ there exists a $t_0=t_0(x,y,\omega)$ such that for all $t\geq t_0$ we have $\varphi(t,\cdot)\cap\Pi_{i=1}^d[x_i,y_i] \neq \emptyset$, whence the minimal weak cc-attractor must be $[0,1]^d$.
\end{proof}
\begin{rem}\label{rem:asConv}
 For later reference, let us note that Lemma \ref{lem:synch} together with the monotonicity of $\psi$ (defined by the SDE in \eqref{eq:RDS1_inv_transform}) even implies the $\P$-a.s.\ convergence of $\vert \psi(t,\cdot)z-\psi(t,\cdot)\bar z\vert $ to 0 for any $z, \bar z \in \R$, from which we can conclude the $\P$-a.s\ convergence of $\vert\varphi^1(t,\cdot)x-\varphi^1(t,\cdot)y\vert$ to 0 for any $x,y \in [0,1]$.
\end{rem}
\begin{proof}[Proof of Theorem \ref{thm:point_RDS1_inv}]
 Define $A:=\{b_{\alpha}\mid \alpha\in\{-1,0,1\}^d\}$ where $b_{\alpha}=P_{\alpha}(b)$. Since $\{b\}$ is an attractor of $\bar\varphi$ on $]0,1[^d$ the invariance of $A$ under $\bar\varphi$ follows recalling that $0$ and $1$ are invariant unter each $\bar \varphi^i$.
 
 Next, observe that Lemma \ref{lem:synch} in particular implies $\bar\varphi^i(t,\bar\vartheta_{-t}\omega)x \rightarrow b^i$, $t \rightarrow \infty$, for $\P$-almost all $\omega \in \Omega$ for each $i=1,\ldots,d$ and any $x \in ]0,1[$. Hence, for any $x \in \intt\Gamma_{\alpha}$ we see that $\bar\varphi^i(t,\bar\vartheta_{-t}\omega)x^i \rightarrow b^i$, $t\rightarrow \infty$, for $\P$-almost all $\omega \in \Omega$, when $\alpha_i=0$ and the invariance of $0$ and $1$ under any $\bar\varphi^i$ yields $\bar\varphi(t,\bar\vartheta_{-t}\omega)x \rightarrow b_{\alpha}$, $t\rightarrow \infty$, for $\P$-almost all $\omega \in \Omega$. Since this holds for any $\alpha \in \{-1,0,1\}^d$ and any point in $\{0,1\}^d$ is a fixed point of $\bar \varphi$ we have proven that $A$ is indeed the  minimal pullback point attractor as well as the 
minimal weak point attractor.
 
 From Remark \ref{rem:asConv} we then conclude that $A$ is also the minimal forward point attractor. 
\end{proof}

\begin{proof}[Proof of Theorem \ref{thm:delta_RDS1}]
Observe that each face $\Gamma_{\alpha}$, and thus also $H^m$ for any $m = 0,\ldots, d$, is strictly $\varphi$-invariant, since (for every $\omega \in \Omega$ and $t \in \R$) each coordinate of $x\mapsto \varphi(t,\omega)x$ is bijective on $[0,1]$.

Now choose $m \leq d-1$ and $\alpha \in \{-1,0,+1\}^d$ with $\Vert \alpha \Vert = d-(m+1)$. Let $B \subset [0,1]^d$ be a compact set with $\dim_H(B) \in [m,m+1[$.  Since $\dim_H(P_{\alpha}(B))<m+1$, given the uniform distribution of $b$ on $]0,1[^d$ we see that $\P(P_{\alpha}(b) \in P_{\alpha}(B))=0$. The compactness of $P_{\alpha}(B)$ then implies the almost sure existence of a $\delta>0$ such that  $P_{\alpha}(B)\subseteq B_{\delta}(P_{\alpha}(b))^c $. Therefore, for $\P$-almost all $\omega \in \Omega$ and every $y \in B$, there exists an $i \in \{1, \ldots, d\}$ such that $\alpha_i = 0$ and $y^i \nin ]b^i-\delta,b^i+\delta[$ and thus
 \begin{align*}
 \dd(\varphi^i(t,\omega)y^i, \{0,1\}) &  \leq \max_{i =1, \ldots,d}\left(\dd(\varphi^i(t,\omega)(b^i-\delta), \{0,1\})+\dd(\varphi^i(t,\omega)(b^i+\delta), \{0,1\})\right) \rightarrow 0
 \end{align*}
 by Lemma \ref{lem:b}. 
 Since the bound does not depend on $y\in B$ itself, we obtain
 \begin{align*}
  \dd(\varphi(t,\omega)B, \underbrace{\{x \in [0,1]^d\mid \exists\, i =1,\ldots,d: \alpha_i = 0 \text{ and } x_i \in \{0,1\}\}}_{=:G_{\alpha}}) \rightarrow 0
 \end{align*}
for $\P$-almost all $\omega \in \Omega$. Since these considerations hold true for any $\alpha \in \{-1,0,+1\}$ such that $\Vert \alpha \Vert = d-(m+1)$, we have indeed proven that
 \begin{align}\label{eq:G_alpha}
  \dd(\varphi(t,\omega)B, \bigcap_{\alpha: \Vert \alpha \Vert = d - (m+1)} G_{\alpha}) \rightarrow 0
 \end{align}
for $\P$-almost all $\omega \in \Omega$. A moment of thought reveals that 
\begin{align*}
 \bigcap_{\alpha: \Vert \alpha \Vert = d - (m+1)} G_{\alpha} \;= \bigcup_{\alpha: \Vert \alpha \Vert = d - m} \Gamma_{\alpha} \;=\; H^m
\end{align*}
whence \eqref{eq:G_alpha}  proves that $H^m$ is indeed a forward (and a weak) attractor for all compact sets of Hausdorff dimension strictly less than $m+1$. 

Since $\dim_H(H^m) = m$ (and $H^m$ is compact) the strict $\varphi$-invariance of the faces implies that the \emph{minimal} forward or weak $\Delta$-attractor for $\Delta <m+1$, must in particular contain $H^m$, which completes the proof. 
 \end{proof}
 

\section{The second example} \label{sec:Ex2}

Our second example arises naturally in the context of population genetics as a RDS generated by so-called \emph{Volterra polynomial stochastic operators}. The special case of Volterra quadratic 
stochastic operators has been treated in \cite{MyVolterra} while the case of general $d$ is contained in \cite{MyDiss}. 

\subsection{Volterra Polynomial Stochastic Operators}
Quadratic stochastic operators (QSOs) were introduced by Bernstein in \cite{Bernstein} to describe the evolution in discrete generations  of a large population with a given law of heredity. Their theory has since then developed motivated both by their natural and frequent occurrence in mathematical models of genetics as well as their interesting behaviour as mathematical objects in their own right. See \cite{RandomQSOs} for a comprehensive account of QSOs. In some applications, also \emph{cubic} stochastic operators have been considered (cf.\ \cite{Cubic3, Cubic2}, for example), but it is natural for a mathematician to generalize this notion to a \emph{polynomial} stochastic operator (PSO). 

The purpose of PSOs in their biological origin is the description of the evolution of gene-frequencies of a population according to certain hereditarian laws, hence they are often referred to as \emph{evolutionary operators}. We assume that reproduction takes place in non-overlapping (i.e.\@ discrete) generations and each individual has $d \geq 2$ parents from the previous generation. There are $m \in \N$ possible \emph{types} of individuals in the population and each individual belongs to exactly one of these types. (In particular, there is no mutation, i.e.\@ no new types can appear.) Furthermore, the size of the population is taken to be sufficiently large for random fluctuations to have no impact on the frequencies of the types under consideration. The population is in a state of \emph{panmixia}, i.e.\@ the types are assumed statistically independent for breeding, implying the absence of selection and sexual differentiation between the \lq parents'. 
We will call $\mm:= \{1, \ldots, m\}$ the \emph{type-space}. Since we trace the evolution of gene-frequencies, our state space is
\begin{align*}
 S^{m-1}:=\left\{x=(x_1, \ldots,x_m) \in [0,1]^m \mid \sum_{i=1}^m x_i = 1\right\}
\end{align*}
the simplex of probability distributions on $\mm$. We will use $\intt S^{m-1}:=\{x \in S^{m-1}\mid \forall i =1, \ldots, m:\; x_i>0\}$ for the interior as well as $\partial S^{m-1}:=S^{m-1}\setminus \intt S^{m-1}$ for the boundary of $S^{m-1}$ and denote its vertices by $e_1, \ldots, e_m$. 

\begin{defn}\label{def:pso}
For $d \in \N$ let $V:  S^{m-1} \rightarrow S^{m-1}$ be such that for any $x \in S^{m-1}$ and $ k \in \mm$ 
\begin{align}\label{eq:pso}
 (Vx)_k := \sum_{i_d = 1}^{m} \cdots \sum_{i_1 = 1}^{m} p^{i_1,\ldots,i_d}_kx_{i_1}\cdots x_{i_d}
\end{align}
where the coefficients $(p^{i_1,\ldots,i_d}_k)_{i_1,\ldots,i_d,k \in \mm}$ satisfy
\begin{itemize}
 \item[(a)] for any permutation $\pi$ of $\{1, \ldots, d\}$:\, $p^{i_1,\ldots,i_d}_k = p^{i_{\pi(1)},\ldots,i_{\pi(d)}}_k \geq 0$ and
 \item[(b)] $\sum_{k = 1}^{m} p^{i_1,\ldots,i_d}_k = 1$
\end{itemize}
for all $i_1,\ldots,i_d,k \in \mm$. Then $V$ is called a \emph{polynomial stochastic operator (PSO) of degree $d$}. For $d=2$ and $d=3$ we say \emph{quadratic}, respectively  \emph{cubic}, stochastic operator.

If, in addition, for all $i_1,\ldots,i_d,k \in \mm$,
\begin{align}\label{eq:psoVolterra}
 k \nin \{i_1, \ldots, i_d\} \Rightarrow p^{i_1,\ldots,i_d}_k = 0
\end{align}
 then said operator is called a \emph{Volterra} polynomial stochastic operator (VPSO).
\end{defn}
Note that the terms and definitions above have clear biological interpretations:  For $i_1,\ldots,i_d, k \in \mm$ the quantities $p^{i_1,\ldots,i_d}_k$, called \emph{heredity coefficients}, give the conditional probabilities, that $d$ individuals of respective types $i_1, \ldots, i_d$ interbreed to produce an individual of type $k$, given that they meet. Condition (a) is the lack of differentiation between the parents (\lq mother', \lq father', \lq other'...) and (b) comes from the exclusion of mutation. Likewise, the product-structure in \eqref{eq:pso} corresponds to the assumption of statistical independence of types for breeding. The property \eqref{eq:psoVolterra} gives the biologically sensible condition that the offspring must be of the type of one of its parents.

It is a straightforward observation, that PSOs are \emph{continuous}. For the case of $d=2$ it was shown in \cite{VolterraTournaments}, that any Volterra \emph{quadratic} stochastic operator is a homeomorphism.

A Volterra PSO of degree $d$ has the general form 
  \begin{align*}
   (Vx)_k & = \,x_k\,\, \Big[ \;p^{k,\ldots,k}_kx_k^{d-1} + \sum_{l=1}^{d-1} x_k^{d-l-1} \binom{d}{d-l}\sum_{\substack{i_{l}=1\\ i_{l}\neq k}}^{m} \cdots \sum_{\substack{i_1=1\\ i_1\neq k}}^{m} p^{i_1,\ldots,i_{l},k}_kx_{i_1}\cdots x_{i_{l}}\Big]
  \end{align*}
  for any $x \in S^{m-1}$ and $k \in \mm$. In particular $(Vx)=0$ whenever $x_k=0$, hence $e_1, \ldots, e_m$ are fixed points. Likewise, any face $\Gamma_{\alpha}:=\{ x \in S^{m-1} \mid x_i \leq \alpha_i\}$ (for any $\alpha \in \{0,1\}^m$) remains invariant in the sense that $V(\Gamma_{\alpha}) = \Gamma_{\alpha}$.

For our further considerations we need to introduce a subclass of VPSOs, that correspond to hereditary mechanisms that impose stricter conditions on the reproduction of specific types. 

\begin{defn}\label{def:purebred}
In the hereditary mechanisms described by a PSO $V$ of degree $d$ with heredity coefficients $(p^{i_1, \ldots, i_d}_k)_{i_1, \ldots, i_d,k \in \mm}$ we will call a type $k \in \mm$ \emph{purebred} if
 \begin{align}\label{eq:conditionPureBredk}
   \vert\left\{j \in \{1, \ldots, d\} \mid i_j \neq k\right\}\vert \geq d-2\;\Rightarrow\; p^{i_1, \ldots, i_d}_k = 0
 \end{align}
for all $i_1, \ldots, i_d,k \in \mm$. 
\end{defn}
In biological terms, a purebred is then a type that can only occur if at least \emph{two} of its parents are of this type. For QSOs this means that \emph{all} parents need to be of the same type, which explains the terminology. If all types are purebred in a PSO, then it is also a Volterra PSO and a moment's thought reveals that for every type $k \in \mm$ we can easily find a Volterra PSO $V$ such that $k$ is purebread in $V$.

It is quite obvious from a biological point of view that being \emph{purebred} is a disadvantage when it comes to survival of the type. This is also clearly seen in the mathematical expressions, as simple calculations reveal the following estimates for our two types of PSOs:

\begin{prop}\label{prop:estimatePBV}
 Let $V: S^{m-1} \rightarrow S^{m-1}$ be a PSO of degree $d\geq 2$. If the type $k \in \mm$ is purebred, then for all $x \in S^{m-1}$
 \begin{align*}
  (Vx)_k \leq x_k^2\binom{d}{2}.
 \end{align*}
 If $V$ is a Volterra PSO, then for every type $k \in \mm$ and any initial distribution $x \in S^{m-1}$:
 \begin{align*}
  (Vx)_k \leq dx_k.
 \end{align*}
\end{prop}

The second estimate was observed for VQSOs in \cite{MyVolterra}. This proposition is a crucial observation for this section as it gives uniform bounds on compact subsets of $\intt S^{m-1}$ describing that such sets will receive a strong 'push' in one direction by a VPSO where a type is purebred and that this effect is not easily countered by any other VPSO.

\subsection{Evolution through VPSOs}
 
 Let $\mathcal V$ be the set of all Volterra PSOs of some fixed degree $d\geq 2$. Traditionally, (V)PSOs and their properties have been analyzed in a deterministic context. However, it seems rather natural to randomize the trajectories. We follow the notions and set-up used in \cite{MyVolterra, RandomQSOs}. 
 
 In order to obtain the RDS, we introduce more structure on this set. The correspondence between the PSOs $V$ and \lq matrices'  $(p^{i_1, \ldots, i_d}_k)_{i_1, \ldots, i_d, k \in \mm}$ gives a natural embedding $\iota: \mathcal V \hookrightarrow \R^{\left(m^{d+1}\right)}$ allowing us to define $\Upsilon:=\iota^{-1}\left(\left\{B\cap\iota(\mathcal V)\mid B \in \mathcal B(\R^{\left(m^{d+1}\right)})\right\}\right)$ as $\sigma$-algebra on $\mathcal V$. For any type $k \in \mm$, let $\mathcal V_k$ be the set of all Volterra PSOs $V$ (of degree $d$) such that the type $k$ is purebred in $ V$. 
These sets are all non-empty and clearly measurable with respect to $\Upsilon$. Let $\nu$ be a measure on $(\mathcal V,\Upsilon)$ such that $\nu_k:=\nu(\mathcal V_k)>0$ for all types $k \in \mm$.

We are now ready to define the RDS to be considered in this section.
\begin{defn}\label{def:ourRDS1}
 Let $d\geq 2$. Let the probability space be given by $ \Omega := \mathcal V^{\Z}$, $\mathcal F:=\Upsilon^{\otimes\Z}$, $\P:=\nu^{\otimes\Z}$. We then define $\vartheta: \Omega \rightarrow \Omega$ by $(\vartheta\omega)_i:=\omega_{i+1}$ for all $\omega \in \Omega$, $i \in \Z$ and with it the family $\vartheta_z:=\vartheta^z$, $z \in \Z$ which completes the MDS. Consider $S^{m-1}$ with the metric $\dd$ induced by the euclidian norm. We define the RDS $\varphi: \N_0 \times \Omega\times S^{m-1}  \rightarrow S^{m-1}$
 as
 \begin{align}\label{eq:defRDSd}
  \varphi(n,\omega)x := \begin{cases}
			    \omega_n \circ \cdots \circ \omega_1 x,	& \text{ for } n \in \N,\\
			    x, 		& \text{ for } n = 0.
                           \end{cases}
 \end{align}
In the case of $d=2$ we can, in addition, set $\varphi(z,\omega)x:= \omega^{-1}_{z+1}\circ \cdots\circ \omega^{-1}_0x$ for $z \in -\N$. 
\end{defn}
The construction in \eqref{eq:defRDSd} is the standard construction for RDS with discrete one-sided, respectively two-sided time $ \mathbb T_2$, cf. Section 2.1 in \cite{Arnold_book}. The necessary measurability of  $(\omega,x)\mapsto\varphi(1,\omega)x$ follows from the measurability of  $\omega \mapsto \varphi(1,\omega)x$ and the continuity of $x \mapsto \varphi(1,\omega)x$ for any $\omega \in \Omega$ by Lemma 1.1 in \cite{Crauel_book}. Due to the structure of $\P$, the RDS defined above is a product of random mappings, i.e. it has independent increments, cf. \cite{Arnold_book} Section 2.1.3. Hence, for any $n \in \N$, the $\sigma$-algebras $\mathcal F_n:=\sigma(\varphi(k,\,\cdot\,)\mid k \leq n)$ and $\vartheta_{-n}\mathcal F$ are independent. 

Since $\varphi$ are homeomorphisms in the case of $d=2$, the  forward and pullback attractor must be $S^{m-1}$ itself. It was proven in \cite{MyVolterra} for QVSOs that the minimal  forward and pullback point attractor is given by $\Lambda:=\{e_1, \ldots, e_m\}$. The analogous result for VPSOs can be found in \cite{MyDiss}. However, from the bounds in Proposition \ref{prop:estimatePBV} one can deduce that any compact set in $\intt S^{m-1}$ would be uniformly pushed towards $\partial S^{m-1}$ and since the faces of $S^{m-1}$ are nothing but simplices of smaller dimension the question of the relation between the attractor for sets of a certain Hausdorff dimension $\Delta$ and the faces of $S^{m-1}$ of dimension $\floor{\Delta}$ arises naturally in this context. Here, we take a first step in this direction and prove that $\Lambda$ is not only the (forward) point attractor, as proven in \cite{MyVolterra}, but also a  forward $\Delta$-attractor for some $\Delta>0$, which is the main result of this section.
\begin{thm}\label{thm:deltaatt}
 There exists a $\Delta>0$ such that $\Lambda=\{e_1,\ldots,e_m\}$ is the minimal  forward $\Delta$-attractor for the RDS $\varphi$ on $S^{m-1}$.
\end{thm}

\begin{rem}
 Theorem \ref{thm:deltaatt} holds for every $\Delta <\beta$, for $\beta$ given by \eqref{eq:beta} in Remark \ref{not:crazylist}. Unfortunately the precise value of an optimal $\beta$ obtainable in the context of this proof  is not easy to come by even for special cases. At the end of the section we provide an exemplary choice of parameters that shows that $\beta \geq 1.54*10^{-4}$ for $m=d=2$ and $\nu = (1/2,1/2)$ the uniform distribution.
\end{rem}

Let us briefly discuss the strategy of the proof before indulging in the collection of preliminary results. Note that a set $H$ must be attracted by $\Lambda$, if we can find a cover of $H$ that is attracted by $\Lambda$. Since these covers are made up of (very small) sets, the bulk of work is in proving that such small sets converge to $\Lambda$ with high probability. This is done in two steps: Proposition \ref{prop:diamondsconverge} shows that a small neighborhood of $\Lambda$ will converge uniformly to $\Lambda$. We then have to invest additional work to guarantee that the small cover sets reach this neighborhood and stay sufficiently small to be completely contained in it, which is the content of Corollary \ref{cor:reachdiamonds}. A careful combination of these then yields the desired result.

For any $i \in \mm$ and $h \in [0,1]$ we define the sets
 \begin{align*}
  D^i_h:=\left\{x \in S^{m-1} \mid x_i \leq h\right\},\qquad 
  \bar U^i_h	 := \bigcap_{j \in \mm\setminus \{i\}} D^j_h \qquad \text{and} \qquad 
 \bar U_h 	 := \bigcup_{i \in \mm} \bar U^i_h.
 \end{align*}

We estimate the probability of the latter converging to $\Lambda$ uniformly under the action of $\varphi$. Recall that the measure $\nu$ was such that $\underline{\nu}:=\min\{\nu_1, \ldots, \nu_m\}>0$.

\begin{prop}\label{prop:diamondsconverge}
 For any $h \in [0,1]$
 \begin{align*}
  \P(\lim_{n \rightarrow \infty} \dd(\varphi(n,\,\cdot\,)\bar U_h, \Lambda) = 0 ) \;\geq\;1-m\kappa^{-\alpha_1}h^{\alpha_1}
 \end{align*}
 if we choose
  \begin{align*}
   0<\alpha_1 < \frac{-\log(1-\underline{\nu})}{\log(d)} \qquad \text{and} \qquad \kappa:=\min\left\{\left(\frac{1-(1-\underline{\nu})d^{\alpha_1}}{\underline{\nu}\binom{d}{2}^{\alpha_1}}\right)^{\frac{1}{\alpha_1}}, d^{-1}\right\}.
  \end{align*}
\end{prop}

\begin{proof} For any set $C \subseteq S^{m-1}$ we define its `height' with respect to the coordinate $k \in \mm$ by $\text{hei}_k(C):=\sup\{x_k\mid x \in C\}$.
The key element of this proof are the bounds in Proposition  \ref{prop:estimatePBV} that become
 \begin{align*}
\text{hei}_k(VC) & \leq \binom{d}{2}\text{hei}_k(C)^2, \qquad \text{ if }  V\in \mathcal V_k \text{ and }\\
		  \text{hei}_k(VC) & \leq d \text{hei}_k(C), \hspace{40pt} \text{ if } V \in \mathcal V 
 \end{align*}
for every $k \in \mm$.
Without loss of generality assume $k=1$ and fix some value $h \in [0,\kappa]$ (the assertion of the proposition holds trivially for $h\geq\kappa$. 

For $\omega \in \Omega$ and $n \in \N_0$ set $H_0(\omega):= \text{hei}_1(D^1_{h})=h$ and
\begin{align*}
 H_{n+1}(\omega) &:= \begin{cases}
                     \binom{d}{2}H_n^2(\omega)\wedge 1, & \text{ if } \omega_{n+1} \in \mathcal V_1\\
                     dH_n(\omega)\wedge1, & \text{ otherwise.}
                    \end{cases}
\end{align*}
This defines a time-homogeneous Markov chain $(H_n)_{n \in \N_0}$ on $[0,1]$ with the following transition probabilities: If the unique $s^*$ such that $\binom{d}{2}(s^*)^2=ds^*$ is in $[0,1]$, we have
  \begin{align*}
  \P(H_2=\binom{d}{2} (s^*)^2\wedge1\mid H_1= s^*)\; & = \;\; \P(H_2=d  s^*\wedge1\mid H_1 = s^*)=1 \;\;\text{and} \\
  \P(H_2=\binom{d}{2} s^2\wedge1\mid H_1= s)\; & = \;\;1-\P(H_2=d  s\wedge1\mid H_1 = s)\; = \;\;\nu(\mathcal V_1)>0
 \end{align*}
 for all other $s \in [0,1]\setminus\{s^*\}$ in any case. This Markov chain is conveniently coupled to our RDS: 
 \begin{align*}
 \forall \omega \in \Omega,\;  n \in \N_0: \qquad \text{hei}_1(\varphi(n,\omega)D^1_h) \leq H_n(\omega).  
 \end{align*}
 Hence, the long-term behaviour of $(H_n)_{n \in \N_0}$ should be analyzed in more detail: Let $\alpha_1$ and $\kappa$ be as in the assumptions and note that they were chosen such that for any $s \leq \kappa$, we have $\binom{d}{2}s^{2}\wedge 1 = \binom{d}{2}s^{2}$, $ds\wedge1=ds$ and
 \begin{align}\label{eq:cleverkappaalpha}
  \underline{\nu}\binom{d}{2}^{\alpha_1}s^{\alpha_1} + (1-\underline{\nu})d^{\alpha_1} \leq 1.
 \end{align}
Define $\tau:=\inf\{n \in \N_0 \mid H_n >\kappa\}$ and with it the stopped process $\tilde H_n:=\min\{H_{n \wedge\tau},\kappa\}$ (for all $n \in \N_0$). $(\tilde H_n)_{n \in \N_0}$ is then again a time-homogeneous Markov chain. Furthermore, define $v:[0,1]\rightarrow [0,1]$ as $v(s):=s^{\alpha_1}$. Then $\E[v(\tilde H_{n+1}) \mid \mathcal F_n] = \kappa^{\alpha_1}$ on $\{\tau \leq n\}$ and
\begin{align*}
 \E[v(\tilde H_{n+1}) \mid \mathcal F_n]	
						& = \tilde H_n^{\alpha_1}\left(\underline{\nu}\binom{d}{2}^{\alpha_1}\tilde H_n^{\alpha_1} + (1-\underline{\nu})d^{\alpha_1}\right) \overset{\eqref{eq:cleverkappaalpha}}{\leq} \tilde H_n^{\alpha_1} = v(\tilde H_n)
\end{align*}
on $\{\tau>n\}$ for any $ n \in \N$. (Observe that these calculations include the case of $\{\tilde H_n = s^*\}$.)
 Therefore, $(v(\tilde H_n))_{n \in \N_0}$ is a bounded supermartingale and converges $\P$-almost surely to some (random variable) $v_{\infty}$. This implies that $(\tilde H_n)_{n \in \N_0}$ then converges to $H_{\infty}=v_{\infty}^{1/\alpha_1}\in \{0,\kappa\}$
and $ \P(v_{\infty} = 0) = 1- \P(v_{\infty}=\kappa^{\alpha_1})$.
Since $(v(\tilde H_n))_{n \in \N_0}$ is a supermartingale  and $h \leq \kappa$ we can estimate
\begin{align*}
 h^{\alpha_1} 		& = \E[v(\tilde H_0)] \geq \E[v_{\infty}]  =  \kappa^{\alpha_1} \P(v_{\infty}=\kappa^{\alpha_1})
 \end{align*}
 which implies $\P(v_{\infty} = 0) \geq 1-\kappa^{-\alpha_1}h^{\alpha_1}$. The coupling between $(H_n)_{n \in \N_0}$ and $\varphi$ then implies 
$ \P(\lim_{n \rightarrow \infty} \text{hei}_1(\varphi(n,\cdot)D^1_h = 0)  \geq 1-\kappa^{-\alpha_1}h^{\alpha_1}$.

 Since this argument holds for any $k \in \mm$ we have proven
 \begin{align*}
  \P\left(\forall\,i \in \mm: \; \lim_{n\rightarrow \infty} \text{hei}_i(\varphi(n,\,\cdot\,)D^i_h) = 0\right) \geq 1-m\kappa^{-\alpha_1}h^{\alpha_1}.
 \end{align*}
 Given that $\bar U^j_h \subset D^i_h$ for all $i \in \mm\setminus\{j\}$, this implies
  \begin{align*}
   \P(\lim_{n \rightarrow \infty} & \dd(\varphi(n,\,\cdot\,)\bar U_h, \Lambda) = 0 ) \geq\P(\forall\,j\in \mm: \; \lim_{n \rightarrow \infty} \dd(\varphi(n,\,\omega\,)\bar U^j_h,e_j) =0) \geq 1-m\kappa^{-\alpha_1}h^{\alpha_1}
 \end{align*}
 which completes the proof.
\end{proof}

Now that we have a result on the uniform convergence of sets sufficiently `close to' $\Lambda$, we are left to assure that the RDS reaches these sets `fast enough'. To this end, consider the following partition of $S^{m-1}$:
\begin{align*}
 Q_0	 := S^{m-1}\setminus \bar U_{d^{-l_0}d^{-1}},  \qquad  Q_l	:= \bar U_{d^{-l_0}d^{-l}}\setminus \bar U_{d^{-l_0}d^{-(l+1)}}, \qquad l \in \N,
\end{align*}
where the value of $l_0$ is assigned in Remark \ref{not:crazylist} below. In a slight abuse of notation, define a map $l:\,S^{m-1} \rightarrow \N_0$ as $ l(x) \,:=\,l$ if $x \in Q_l$.
Note that then $ x \in \bar U_{d^{-l_0}d^{-l}}$ if and only if  $ l(x) \geq l$. 
The idea is to characterize the convergence of $(\varphi(n,\,\cdot\,)x)_{n\in \N_0}$  through $(l(\varphi(n,\,\cdot\,)x))_{n\in \N_0}$, since
\begin{align*}
 \lim_{n \rightarrow \infty} \varphi(n,\,\cdot\,)x \in \Lambda \qquad \Leftrightarrow \qquad \lim_{n \rightarrow \infty}l(\varphi(n,\,\cdot\,)x) = \infty.
\end{align*}

 For this purpose, we construct a Markov chain $(L_n)_{n \in \N_0}$ that dominates $(l(\varphi(n,\,\cdot\,)x))_{n\in \N_0}$ in the following sense: Fix $\gamma >0$. For $x \in S^{m-1}$, let $N \in \N$ set
\begin{align*}
 \sigma_{\gamma}(x,N)	:& = \inf\left\{ n \in \N_0 \mid \varphi(n,\,\cdot\,)x \in \bar U_{d^{-l_0}d^{-\gamma N}}\right\} = \inf\left\{ n \in \N_0 \mid l(\varphi(n,\,\cdot\,)x) \geq \gamma N\right\},
\end{align*}
and similarly $\tau_{\gamma}(N)	: = \inf\left\{n \in \N_0 \mid L_n \geq \gamma N\right\}$.
Then $(L_n)_{n \in \N_0}$ is intended to be such that 
\begin{align}\label{eq:relationLandphi}
 \P(\sigma_{\gamma}(x,N) \leq N) \;\;  \geq \;\;  \P(\tau_{\gamma}(N) \leq N \mid L_0 = 0)
\end{align}
for every $x \in S^{m-1}$ and $N \in \N$. To this end, we need to construct $(L_n)_{n \in \N_0}$ with a weaker drift towards $\infty$ than $(l(\varphi(n,\,\cdot\,)x))_{n\in \N_0}$, but sufficiently strong to still reach high levels in a `short' amount of steps.

What follows is a list of all parameters (and respective conditions) used in the forthcoming deliberations. We invite the reader to skip it and only return to it for reference, when the parameters appear further on.

\begin{rem}\label{not:crazylist}
Recall that the measure $\nu$ on $\mathcal V$ was chosen such that $\underline{\nu} := \min\{\nu_1, \ldots, \nu_m\}>0$.
 With this then let
  \begin{align*}
   \alpha_1 \in \left]\,0, \frac{-\log(1-\underline{\nu})}{\log(d)}\right[ \quad \text{and set} \quad \kappa:=\min\left\{\left(\frac{1-(1-\underline{\nu})d^{\alpha_1}}{\underline{\nu}\binom{d}{2}^{\alpha_1}}\right)^{\frac{1}{\alpha_1}}, d^{-1}\right\}
   \end{align*}
   as in Proposition \ref{prop:diamondsconverge}. Then define $p:=\underline{\nu}^{m-1} (>0)$.
   
 In the case of $m=2$, let $\mu_2:=0$, otherwise choose any $  \mu_2 \in ]0,1[ $ 
 and subsequently $ \lambda \in\left]0\,,\, -\log(1-p)/(m-1+\mu_2)\right[$.
 Note that this was choses such that
 \begin{align}\label{eq:lambdaprop}
  e^{\lambda(m-1+\mu_2)}(1-p) <1.
 \end{align}
 Choose
 \begin{align*}
  l_1 > -\frac{1}{\lambda}\log\left(\frac{1-e^{\lambda(m-1+\mu_2)}(1-p)}{p}\right) \quad (>0 \text{ by } \eqref{eq:lambdaprop})
 \end{align*}
 and note that this implies
 \begin{align}\label{eq:l1prop}
   e^{\lambda(m-1+\mu_2)}(1-p) + e^{-\lambda l_1}p < 1.
 \end{align}
 Define $ l_0:=\ceil{l_1-1+\mu_2+2(m-1)}$ as well as
 \begin{align*}
  M 	
	& := (m-1)\left\lceil\frac{2}{\log(2)}\log\left(\frac{2\log(d)}{\log(2)^2}\max\left\{m,l_0+1\right\}\right)\right\rceil-1.
 \end{align*}
 Set $ q := \underline{\nu}^{M+1} $ to then choose $  \mu \in\left]0\,,\,-\frac{1}{\lambda}\log\left(1-q+e^{-\lambda}q\right)\right[\;\subset \; ]0,1[$ in order for
 \begin{align}\label{eq:muprop}
  e^{\lambda \mu}(1-q)+e^{-\lambda(1-\mu)}q <1
  \end{align}
   to hold. Set
  \begin{align*}
  A	& := e^{-\frac{\lambda\mu}{M}} \;\;(<1),\\
  B 	& := e^{\lambda \mu}(1-q)+e^{-\lambda(1-\mu)}q <1, \text{ by \eqref{eq:muprop}},\\
  C	& := e^{-\lambda\frac{\mu_2}{m-2}}\;\; (<1), \text{ for } m\geq 3 \text{ and } C:= 0, \text{ for }m=2,\\
  D	& := e^{\lambda(m-1+\mu_2)}(1-p) + e^{-\lambda l_1}p<1, \text{ by \eqref{eq:l1prop}},\\
  E	& := \max\{A,B,C,D\}\;\;(<1).
  \end{align*}  
  Furthermore, choose $\gamma \in \left]0, -\log(E)/\lambda\right[$  and observe that this was chosen such that 

 \begin{align*}
  \alpha_2:= -(\lambda\gamma + \log(E))  \quad \text{and therefore} \quad \alpha_3:=\min\{\alpha_1,\alpha_2\}
 \end{align*}
 is positive. Finally, define
  \begin{align}\label{eq:beta} 
  c:= \exp\left(\frac{-\alpha_3l_0\log(d)}{(1+\gamma+\frac{\log(m)}{\log(d)})}\right) \quad \text{and}\quad \beta :=\frac{\alpha_3}{(1+\gamma+\frac{\log(m)}{\log(d)})}\;\;(>0).
  \end{align}
\end{rem}

Now we are ready to define the Markov chain that is to, in a sense, dominate our RDS.

\begin{defn}\label{def:MCMC}
 For the parameter choice given in Remark \ref{not:crazylist}, let $(L_n)_{n \in \N_0}$ be the time-homogeneous Markov chain (on the probability space $(\Omega,\mathcal F,\P)$) with state space
 \begin{align*}
  \Xi:=\left\{\;i\frac{\mu}{M} \;\Big\mid\; i = 0, \ldots, M\;\right\} \cup \bigcup_{l \in \N}\left\{\;l+i\frac{\mu_2}{m-2}\;\Big\mid i = 0, \ldots, m-2\right\}
 \end{align*}
 (using the convention of $0\cdot \infty = 0$) defined by the following transition probabilities: For any $a, b \in \Xi$ with $a\neq b$
 \begin{align*}
   \P(L_1 = b \mid L_0=a) = \begin{cases}
                             1,		& \text{ if } b-a = \frac{\mu}{M} \text{ (and } a<1\text{)},\\
                             q,		& \text{ if } a = \mu \text{ and } b = 1,\\
                             1-q,	& \text{ if } a = \mu \text{ and } b = 0,\\
                             1,		& \text{ if } b-a = \frac{\mu_2}{m-2} \text{ (and } a\geq 1, b\nin \N\text{)}, \star\\
                             p,		& \text{ if } a -\mu_2\in \N \text{ and } b = 2(a-\mu_2)-2(m-1)+l_0\\
                             1-p,	& \text{ if } a -\mu_2\in \N \text{ and } b = \max\{a-\mu_2-(m-1),0\}.
                            \end{cases}
 \end{align*}
\end{defn}
Note that, since we assumed $a\neq b$ and $\mu_2=0$ in the case of $m=2$, the transition marked by $\star$ is only possible for $m\geq 3$. 
\begin{prop}\label{prop:convMCMC}
 Let $(L_n)_{n \in \N_0}$  be the Markov chain defined above. 
 
 \noindent For every $N \in \N$ we then have
 \begin{align*}
  \P(L_N \geq \gamma N \mid L_0=0) \; \geq \; 1-e^{-\alpha_2N}.
 \end{align*}
\end{prop}

Before we turn to the proof of this proposition, let us explain the connection between this Markov chain $(L_n)_{n \in \N_0}$ and $(l(\varphi(n,\,\cdot\,)x))_{n \in \N_0}$. This is best done in two steps, each corresponding to one area of the state space of $(L_n)_{n \in \N_0}$.

\paragraph{Step 1:} \textbf{ Behavior on $\bigcup_{l \in \N}\left\{\;l+i\frac{\mu_2}{m-2}\;\Big\mid i = 0, \ldots, m-2\right\}$.}

Let $ x \in \bar U_{d^{-l_0}d^{-l}}$, i.e. $l(x) \geq l$ for some $l \in \N$ arbitrary but fixed. Without loss of generality assume $x \in \bar U^m_{d^{-l_0}d^{-l}}$, i.e. for all $ i =1, \ldots, m-1$: $x_i \leq d^{-l_0}d^{-l}$.
Recall that by Proposition \ref{prop:estimatePBV} 
 \begin{align*}
h_k(VC)  \leq \binom{d}{2}h_k(C)^2, \; \text{ if }  V\in \mathcal V_k \quad  \text{ and }\quad 
		  h_k(VC)  \leq d h_k(C), \; \text{ if } V \in \mathcal V 
 \end{align*}
for every $k \in \mm$.
 Now choose $V_i\in \mathcal V_i$ for each $i = 1, \ldots, m-1$. If we apply these (in any order) to $x$, each of the first $m-1$ coordinates will be multiplied by $d$ $m-2$ times and squared once. The worst case occurs, when the squaring step is the last one, hence we can estimate
 \begin{align*}
  (V_{m-1}\circ\cdots\circ V_1x)_i 	& \leq \binom{d}{2}\left(d^{-l_0}d^{-l}d^{m-2}\right)^2 =d^{-l_0}d^{-(l_0 - 2(m-1) + 2l)}.
 \end{align*}
 Therefore, we know $  V_{m-1}\circ\cdots\circ V_1x \;\in\; \bar U_{d^{-l_0}d^{-(l_0 - 2(m-1) + 2l)}}$,
 which is equivalent to $l(V_{m-1}\circ\cdots\circ V_1x) \geq l_0 - 2(m-1) + 2l$. The probability of such a sequence of operators appearing consecutively is at least $\underline{\nu}^{m-1}$. Of course, $(\varphi(n,\,\cdot\,)x))_{n \in \N_0}$ might also move closer to $e_m$ with other combinations of operators (in particular with any permutation of the above), and thus the probability for $(l(\varphi(n,\,\cdot\,)x))_{n \in \N_0}$ to jump from $l$ to a value above $l_0 - 2(m-1) + 2l$ in $(m-1)$ steps is greater that the probability of $(L_n)_{n \in \N_0}$ to jump from $l$ to (exactly) $l_0 - 2(m-1) + 2l$. 

 On the other hand, we also have the following estimate for any sequence of $V_1, \ldots, V_{m-1} \in \mathcal V$:
 \begin{align*}
  (V_{m-1}\circ\cdots\circ V_1x)_i 	& \leq d^{-l_0}d^{-l}d^{m+1} = d^{-l_0}d^{-(l-(m-1))}
 \end{align*}
 for every $i =1, \ldots,m-1$. This implies $l(V_{m-1}\circ\cdots\circ V_1x ) \geq l-(m-1)$ which the reader will recognize as the complementary jump of $(L_n)_{n \in \N}$. Hence, $(L_n)_{n \in \N_0}$ will also jump farther to the left than $(l(\varphi(n,\,\cdot\,)x))_{n \in \N_0}$ in the case of the complementary event. Together this implies that $(l(\varphi(n,\,\cdot\,)x))_{n \in \N_0}$ has a stronger drift to the right (towards higher values) than $(L_n)_{n \in \N_0}$ on this part of the state space. Note that this reasoning holds true for all $m \geq 2$.
 
 \paragraph{Step 2:} \textbf{Behavior on $\left\{\;i\frac{\mu}{M} \;\Big\mid\; i = 0, \ldots, M\;\right\}$.}
 
The important observation for this part is that $M$ and $q$ were chosen such that 
 \begin{align*}
 \forall\,x \in S^{m-1}:\quad  \P(\varphi(M+1,\,\cdot\,)x \in \bar U_{d^{-l_0}d^{-1}}) \geq q.
 \end{align*}
 This follows in the case of QSOs from Proposition 3.2 in \cite{MyVolterra}. See also Proposition 3.18 and the remark thereafter in \cite{MyDiss} for the general case of PSOs. Hence, for any $x \in S^{m-1}$, the probability for $(l(\varphi(n,\,\cdot\,)x))_{n \in \N_0}$ to reach the value 1 in $(M+1)$ steps is greater or equal to the probability of $(L_n)_{n \in \N_0}$ reaching 1 in $(M+1)$ steps. Since $(\varphi(n,\,\cdot\,)x)_{n \in \N_0}$ is Markovian, we see again that also in this area it has a stronger drift to the right, i.e. the larger values, than $(L_n)_{n \in \N_0}$. 

 Combining the observations in Step 1 and Step 2, we conclude that $(\varphi(n,\,\cdot\,)x)_{n \in \N_0}$ has overall a stronger drift to the right than $(L_n)_{n \in \N_0}$ and \eqref{eq:relationLandphi} holds. Thus Proposition \ref{prop:convMCMC}  yields the following corollary:
 
 \begin{cor}\label{cor:reachdiamonds}
  Let $\gamma$ and $\alpha_2$ be as defined in  Remark \ref{not:crazylist}. Then for every $N \in \N$ and every $x \in S^{m-1}$:
 \begin{align*}
  \P(\,\exists\, k \leq N:\; \varphi(k,\,\cdot\,)x \in \bar U_{d^{-l_0}d^{-\gamma N}}) =  \P(\sigma_{\gamma}(x,N) \leq N) \; \geq \; 1-e^{-\alpha_2N}.
 \end{align*}
 \end{cor}
 
\begin{proof}[Proof of Proposition \ref{prop:convMCMC}]
Abbreviate $\P_0 := \P(\;\cdot\;\mid L_0=0)$. We want to apply a large-deviation-type argument. Let all parameters be as in Remark \ref{not:crazylist}. Let $\mathcal G_n:=\sigma(L_k \mid k \leq n)$ for any $n \in \N_0$.
 For any $N \in \N$, we then have
 \begin{align}\label{eq:LDargument}
  \P_0(L_N < \gamma N)	
			& \leq e^{\gamma\lambda N} \E_0[e^{-\lambda L_N}] = e^{\gamma\lambda N} \E_0[e^{-\lambda \sum_{n=1}^N(L_n - L_{n-1})}]\notag\\
			& = e^{\gamma\lambda N} \E_0[\prod_{n=1}^{N-1} e^{-\lambda (L_n - L_{n-1})}\E_0[ e^{-\lambda (L_N - L_{N-1})}\mid \mathcal G_{N-1}]]
 \end{align}
 We can estimate the conditional expectation making use of the Markov property of $(L_n)_{n \in \N_0}$: Let $n \in \N$.
 On $\{L_{n-1} < \mu\}$:
 \begin{align*}	
  \E_0[ e^{-\lambda (L_n - L_{n-1})}\mid \mathcal G_{n-1}]	= e^{-\lambda(L_{n-1} + \frac{\mu}{M} - L_{n-1}}) = e^{-\lambda\frac{\mu}{M}} =:A<1.
 \end{align*}
 On  $\{L_{n-1} = \mu\}$:
 \begin{align*}
   \E_0[ e^{-\lambda (L_n - L_{n-1})} \mid \mathcal G_{n-1}]	& = e^{\lambda\mu}(1-q) + e^{-\lambda(1-\mu)}q =:B < 1 \quad \text{ by }\eqref{eq:muprop}.
 \end{align*}
 On $\{L_{n-1} \in [l, l+ \mu_2[\}$ for some $l \in \N$:
 \begin{align*}
  \E_0[ e^{-\lambda (L_n - L_{n-1})} \mid \mathcal G_{n-1}]	& = e^{-\lambda(L_{n-1} + \frac{\mu_2}{m-2} - L_{n-1})} e^{-\lambda \frac{\mu_2}{m-2}} =:C<1 .
 \end{align*}
 On $\{L_{n-1} = l+ \mu_2\}$ for some $  l \in \N$
 \begin{align*}
  \E_0[ e^{-\lambda (L_n - L_{n-1})} \mid \mathcal G_{n-1}]	& = e^{-\lambda(\max\{0,l-(m-1)\} - (l+\mu_2))}(1-p)+ e^{-\lambda(l_0-2(m-1)+2l)}p\\
								& \leq  e^{-\lambda(l-(m-1) - (l+\mu_2))}(1-p)+ e^{-\lambda(l_0-2(m-1)+2l)}p\\
								& \leq  e^{\lambda(m-1 +\mu_2)}(1-p)+ e^{-\lambda(l_0-2(m-1)+2)}p\\
								& \leq  e^{\lambda(m-1 +\mu_2)}(1-p)+ e^{-\lambda l_01}p=:D<1\quad \text{ by }\eqref{eq:l1prop}.
  \end{align*}
 Hence, if we define $E:=\max\{A,B,C,D\}$ we have 
 \begin{align*}
  \E_0[ e^{-\lambda (L_n - L_{n-1})} \mid \mathcal G_{n-1}] \leq E \quad \P\text{-almost surely for every }n \in \N.  
 \end{align*}
 If we plug this into \eqref{eq:LDargument} and iterate the argument we obtain
 \begin{align*}
  \P_0(L_N < \gamma N)		
				& \; \leq  e^{\gamma\lambda N} E \E_0[\prod_{n=1}^{N-1} e^{-\lambda (L_n - L_{n-1})}]
   \leq   e^{\gamma\lambda N} E^N =  e^{\gamma\lambda N+\log(E)N} = e^{-\alpha_2N}.
 \end{align*}
 Note that this line of arguments holds for all $m\geq 2$. For $m \geq 3$ this is clear. In the case of $m=2$, $\mu_2=0$, hence $[l, l+ \mu_2[ = \emptyset$, which reflects the fact, that the Markov-chain does not take these additional steps for $m=2$. In this case, $C=0$ by definition and hence correctly does not influence the value of $E$. \end{proof}

We have now completed all the preliminary work and move on to the proof of Theorem \ref{thm:deltaatt}. 
\begin{proof}[Proof of Theorem \ref{thm:deltaatt}]
First note that, if $\Lambda$ is a forward $\Delta$-attractor for some $\Delta>0$, then it is already the minimal (forward) $\Delta$-attractor, since $\dim_H(\Lambda)=0\leq \Delta$ and $\Lambda$ itself is (strictly) invariant under $\varphi$. Since $\Lambda$ is also compact, in order to prove the theorem, we need to show that for all $H \in \mathcal C_{\Delta}$:
\begin{align*}
 \lim_{n \rightarrow \infty}\dd(\varphi(n,\,\cdot\,)H,\Lambda) = 0 \qquad \P\text{-a.s.}
\end{align*}

 Let $H \in \mathcal C_{\Delta}$. As explained after the theorem, the aim is to prove the convergence under the action of $\varphi$ of suitable coverings with probability arbitrarily close to one since this implies the convergence of $H$ itself. The proof is divided in two parts: In part 1 we combine Proposition \ref{prop:diamondsconverge} and Corollary \ref{cor:reachdiamonds} to ensure the convergence to $\Lambda$ of sufficiently small `balls'. In part 2 we then use this result to obtain the convergence of a suitable cover and therefore of $H$ itself. 
 
 \noindent\textbf{ Part 1:} Corollary \ref{cor:reachdiamonds} yields the uniform convergence of a set `sufficiently close' to $\Lambda$, hence, any ball in that set will also converge, but we still have to assure that the ball will indeed reach this set. To this end, we have to translate the results obtained in Proposition \ref{prop:diamondsconverge} for points $x$ into a result for (small) balls around $x$. We denote by $B_h(x)$ the open ball around $x\in S^{m-1}$ with radius $h>0$. It is easy to check that such a ball has bounded growth under the action of the RDS $\varphi$, since the derivative of the latter is bounded. To be precise we have that $\varphi(1,\omega)B_h(x) \;\subseteq\;B_{dmh}(\varphi(1,\omega)x)$  for every $\omega \in \Omega$, $x \in S^{m-1}$ and $h>0$.
 In addition, note that 
 $ x \in \bar U_h$ implies $B_h(x)\subseteq \bar U_{2h}$.
 Combining these two considerations, we note that for every $\omega \in \Omega$, $x \in S^{m-1}$,  $h>0$, and $N \in \N$, 
 \begin{align*}
  \varphi(N,\omega)x \in \bar U_h \;\Rightarrow \;\varphi(N,\omega)B_{h(md)^{-N}}(x) \subseteq B_{h}(\varphi(N,\omega)x) \subseteq \bar U_{2h}.
 \end{align*}
 This allows us to conclude
 \begin{align}\label{eq:bound1}
  \P(\,\exists\, k \leq N:\; \varphi(k,\,\cdot\,)&B_{d^{-l_0}d^{-\gamma N}(dm)^{-N}}(x) \in \bar U_{2d^{-l_0}d^{-\gamma N}}) \notag\\
	   &\geq \P(\,\exists\, k \leq N:\; \varphi(k,\,\cdot\,)x \in \bar U_{d^{-l_0}d^{-\gamma N}}) \geq  1-e^{-\alpha_2N}
 \end{align}
 by Corollary \ref{cor:reachdiamonds}. With $\alpha_1>0$ and $\kappa>0$ chosen as in Remark \ref{not:crazylist}, Proposition \ref{prop:diamondsconverge} yields
 \begin{align}\label{eq:bound2}
   \P(\lim_{n \rightarrow \infty} \dd(\varphi(n,\,\cdot\,)\bar U_{2d^{-l_0}d^{-\gamma N}}, \Lambda) = 0 )
			    & \geq  1-m\kappa^{-\alpha_1}(2d^{-l_0})^{\alpha_1}e^{-\gamma \alpha_1\log(d)N}.
 \end{align}
 These two results can be combined with the following observation:
 \begin{align*}
  \Big\{\omega \in &\; \Omega  \,\Big\mid\lim_{N \rightarrow \infty}\dd(\varphi(n,\omega)B_{d^{-l_0}d^{-\gamma N}(dm)^{-N}},\Lambda)=0\Big\}\\
	    & \supseteq \left\{\omega \in \Omega \,\Big\mid \sigma_{\gamma}(x,N)(\omega)\leq N\right\} \\
	    & \qquad\cap \left\{\omega \in \Omega \,\Big\mid\lim_{n\rightarrow \infty}\dd(\varphi(n-\sigma_{\gamma}(x,N)(\omega),\vartheta_{\sigma_{\gamma}(x,N)(\omega)}\omega)\bar U_{2d^{-l_0}d^{-\gamma N}},\Lambda)=0\right\} \\
	    &= \bigcup_{k=1}^N\Big(\underbrace{\left\{\omega \in \Omega \,\Big\mid \sigma_{\gamma}(x,N)(\omega)=k\right\}}_{\in \mathcal F_k}\cap \underbrace{\vartheta_{-k}\left\{\omega \in \Omega \,\Big\mid\lim_{n\rightarrow \infty}\dd(\varphi(n,\omega)\bar U_{2d^{-l_0}d^{-\gamma N}},\Lambda)=0\right\}}_{\in \vartheta_{-k}\mathcal F} \Big).
 \end{align*}
 Note that we have to shift $\omega$ in the second set, since we want $\bar U_{2d^{-l_0}d^{-\gamma N}}$ to converge only \emph{once the ball has reached it}. Recall that the increments of $\varphi$ are independent under $\P$ and we have that for any $k \in \N_0$ $\mathcal F_k$ and $\vartheta_{-k}\mathcal F$ are independent. In a Markovian manner we  argue that our RDS `restarts' independently of the past once it reaches the set $ \bar U_{2d^{-l_0}d^{-\gamma N}}$. 

Since the union is disjoint independence yields,

 \begin{align*}
  \P\big(\lim_{N \rightarrow \infty} &\dd(\varphi(n,\omega)B_{d^{-l_0}d^{-\gamma N}(dm)^{-N}},\Lambda)=0\big) \\
				     & \geq  \P\big(\sigma_{\gamma}(x,N)\leq N \text{ and } \\
				     & \qquad\quad  \lim_{n\rightarrow \infty}\dd(\varphi(n-\sigma_{\gamma}(x,N)(\omega),\vartheta_{\sigma_{\gamma}(x,N)(\omega)}\omega)\bar U_{2d^{-l_0}d^{-\gamma N}},\Lambda)=0\big)\\
				     & = \P\big(\sigma_{\gamma}(x,N)\leq N \big)\\
				     & \qquad\quad \times \P\big(\lim_{n\rightarrow \infty}\dd(\varphi(n-\sigma_{\gamma}(x,N)(\omega),\vartheta_{\sigma_{\gamma}(x,N)(\omega)}\omega)\bar U_{2d^{-l_0}d^{-\gamma N}},\Lambda)=0\big)\\
				     & \geq (1-e^{-\alpha_2N})(1-m\kappa^{-\alpha_1}(2d^{-l_0})^{\alpha_1}e^{-\gamma \alpha_1\log(d)N})
 \end{align*}
 by \eqref{eq:bound1} and \eqref{eq:bound2}. Hence, we can write
 \begin{align*}
  \P\big(\lim_{N \rightarrow \infty} \dd(\varphi(n,\omega)B_{d^{-l_0}d^{-\gamma N}(dm)^{-N}},\Lambda)=0\big) \;\geq\; 1-e^{-\alpha_3N}
 \end{align*}
 for sufficiently large $N \in \N$, if we set  $ \alpha_3:=\min\{\alpha_1,\alpha_2\}$.
 In order to apply this to a covering, it is more useful to express the probability in terms of the radius of the ball itself. To this end, we rewrite
 \begin{align*}
  r  = d^{-l_0}d^{-\gamma N}(dm)^{-N} = d^{-l_0}d^{(-\gamma - 1 - \frac{\log(m)}{\log(d)})N}\qquad 
  \Leftrightarrow \qquad N  = - \frac{\log(rd^{-l_0})}{(1+\gamma+\frac{\log(m)}{\log(d)})}.
 \end{align*}
 Then, for sufficiently small $r$,
 \begin{align}\label{eq:chooser}
  \P\big(\lim_{N \rightarrow \infty} \dd(\varphi(n,\omega)B_{r},\Lambda)=0\big) \;& \geq\; 1-\exp\left(\alpha_3\frac{\log(rd^{-l_0})}{(1+\gamma+\frac{\log(m)}{\log(d)})}\right)
					     = \;1-c r^{\beta},
 \end{align}
 with 
 \begin{align*}
  c:= \exp\left(\frac{-\alpha_3l_0\log(d)}{(1+\gamma+\frac{\log(m)}{\log(d)})}\right) \qquad \text{ and } \qquad \beta :=\frac{\alpha_3}{(1+\gamma+\frac{\log(m)}{\log(d)})}>0.
 \end{align*}
\noindent \textbf{Part 2:} We are now ready prove the assertion combining the above. Recall that we chose $H \in \mathcal C_{\Delta}$.
The definition of the Hausdorff dimension implies that for every $\delta>0$, $\varepsilon_1>0$ and $\varepsilon_2>0$, there exists a cover (of open sets) $E_1, E_2, \ldots $ of $H$ such that 
 \begin{align*}
  \sum_{i\in \N}\diam(E_i)^{\Delta+\delta} &<\varepsilon_1\\
  \text{and }\quad \forall\, i \in \N:\; \diam(E_i)&<\varepsilon_2.
 \end{align*}
  The remainder of the proof then is, as life, a matter of right choices:  Let $\varepsilon >0$ and choose $\delta:= \beta-\Delta$. Set $\varepsilon_1:=\varepsilon/(c2^{-\beta})$ and choose $r$ in \eqref{eq:chooser} sufficiently small such that $cr^{\beta}< \varepsilon_1$ (and \eqref{eq:chooser} holds). Then set $\varepsilon_2:=2r$. Then there exists a covering (with open sets) $E_1, E_2, \ldots$ with
 \begin{align*}
  \forall i \in \N:\quad \diam(E_i)<\varepsilon_2 \\
  \text{and}\quad \sum_{i \in \N} \diam(E_i)^{\Delta+\delta}<\varepsilon_1.
 \end{align*}
 Then for every $ i \in \N$
 \begin{align*}
  \P(\lim_{n \rightarrow \infty}\dd(\varphi(n,\cdot)E_i,\Lambda)=0) \geq 1-c\left(\frac{\diam(E_i)}{2}\right)^{\beta} = 1-c2^{-\beta}\diam(E_i)^{\beta}
 \end{align*}
 since we chose their diameters sufficiently small for $\diam(E_i)<2r$.
 But this then implies
 \begin{align*}
  \P(\forall\, i \in \N:\; \lim_{n \rightarrow \infty}\dd(\varphi(n,\cdot)E_i,\Lambda)=0) &\geq 1- \sum_{i \in \N}\left(1-\P(\lim_{n \rightarrow \infty}\dd(\varphi(n,\cdot)E_i,\Lambda)=0)\right)\\
			& \geq 1-\sum_{i \in \N} c2^{-\beta}\diam(E_i)^{\beta}\geq 1-c2^{-\beta}\varepsilon_1 = 1-\varepsilon,
 \end{align*}
and therefore
\begin{align*}
 \P(\lim_{n \rightarrow \infty}\dd(\varphi(n,\cdot)H,\Lambda)=0) \geq \P(\forall\, i \in \N:\;& \lim_{n \rightarrow \infty}\dd(\varphi(E_i,\Lambda)=0) \geq 1-\varepsilon.
\end{align*}
Since this can be conducted for any $\varepsilon>0$, we have proven the claim.  
\end{proof}

\begin{rem}
 In the following we give a choice for the parameters of Remark \ref{not:crazylist} in the case $m=d=2$ $\nu =(1/2,1/2)$. For simplicity we only list those parameters whose precise value influences the value of the final $\beta$:
 
 \qquad  \qquad \begin{tabular}{ l l l }
  \qquad $\alpha_1 := 0.99$  		& \qquad \qquad $l_0 := 2$ 		& \qquad \qquad $B := 0.999791$ \\
  \qquad $p := 0.5 $ 			& \qquad \qquad $M := 7$ 		& \qquad \qquad $D := 0.989288$ \\
  \qquad $\mu_2 := 0 $ 			& \qquad \qquad $q := 2^{-8}$ 		& \qquad \qquad $\gamma := 10^{-10}$ \\
  \qquad $\lambda := 4*10^{-1}$ 	& \qquad \qquad $\mu := 0.0027$	& \qquad \qquad $ \alpha_3 := \alpha_2 := 1.54012*10^{-4}$\\
  \qquad $l_1 := 1.8$ 			& \qquad \qquad $A := 0.999846 $ 	& \qquad \qquad $ \beta := 1.54011*10^{-4}$.
\end{tabular}

\end{rem}

\bibliographystyle{alpha}

\bibliography{Bib_delta}

\end{document}